\documentclass[12pt,reqno]{amsart}

\usepackage{amssymb,latexsym}

\usepackage{enumerate}
\allowdisplaybreaks
\usepackage[french,english]{babel}
\usepackage{amsmath}
\usepackage{graphicx}
\usepackage{amssymb}
\usepackage{bbm}
\usepackage{amsthm,mathtools}
\usepackage{ulem}
\usepackage{geometry}
\usepackage{tikz-cd}
\usepackage{mathrsfs}
\usepackage[colorinlistoftodos]{todonotes}
\usepackage{enumitem}
\usepackage{verbatim}
\usepackage[foot]{amsaddr}
\usepackage{dsfont}
\usepackage{cite}
\usepackage[T1]{fontenc}

\makeatletter

\@namedef{subjclassname@2010}{
	
	\textup{2020} Mathematics Subject Classification}

\makeatother
\newtheorem{thm}{Theorem}[section]
\newtheorem*{thm*}{Theorem}

\newtheorem{lem}[thm]{Lemma}

\theoremstyle{definition}

\numberwithin{equation}{section}

\newcommand{\bg}{\big}

\newcommand{\bgg}{\bigg}

\newcommand{\Bg}{\Big}

\newcommand{\lo}{\log_2}
\newcommand{\lt}{\log_3}

\newcommand{\mbr}{\mathbb{R}}

\newcommand{\mcd}{\mathcal{D}}
\newcommand{\mce}{\mathcal{E}}

\newcommand{\mcm}{\mathcal{M}}

\newcommand{\F}{\mathcal{F}}

\newcommand{\mmd}{\mathrm{d}}
\newcommand{\mme}{\mathrm{e}}
\newcommand{\mmi}{\mathrm{i}}

\newcommand{\mdt}{\mathrm{d}t}

\newcommand{\whp}{\widehat{\Phi}}

\usepackage{hyperref}
\hypersetup{colorlinks=true,linkcolor=blue,anchorcolor=blue,citecolor=blue}
\usepackage{color}
\newcommand{\newabstract}[1]{%
	\par\bigskip
	\csname otherlanguage*\endcsname{#1}%
	\csname captions#1\endcsname
	\item[\hskip\labelsep\scshape\abstractname.]
}

\begin{document}

	\baselineskip=17pt

	\title[Z. Dong, W. Wang, H. Zhang and S. Zhao]{Extreme values of quadratic Dirichlet $L$-functions with prime-related moduli}

	\author{Zikang Dong\textsuperscript{1}}
    \author{Weijia Wang\textsuperscript{2}}
    \author{Hao Zhang\textsuperscript{3}}
    \author{Shengbo Zhao\textsuperscript{4}}
    
	\address{1. School of Mathematical Sciences, Soochow University, Suzhou 215006, P. R. China}
    \address{2. School of Mathematics, Shandong University, Jinan 250100, P. R. China}
    \address{3. School of Mathematics, Hunan University, Changsha 410082, P. R. China}
	\address{4. School of Mathematical Sciences, Key Laboratory of Intelligent Computing and Applications (Tongji University), Ministry of Education, Tongji University, Shanghai 200092, China}
	\email{zikangdong@gmail.com}
    \email{weijiawang@amss.ac.cn}
	\email{zhanghaomath@hnu.edu.cn}
	\email{shengbozhao@hotmail.com}

	\begin{abstract} 
	   In this paper, we show a new lower bound for extreme values of quadratic Dirichlet $L$-functions with prime-related moduli, which generalizes the work of Darbar and Maiti in 2025, and sharpens a recent work of Gao in 2026.
	\end{abstract}
	
	
	
	\maketitle

\section{Introduction}

Throughout this paper, we write $\log_j$ for the $j$-th iterated logarithm, such as \(\log_2x= \log\log x\), and \(\log_3 x=\log\log\log x\). Let $q$ always represent an odd prime.

In \cite{soundararajan2008mathann}, Soundararajan established lower bounds for extreme values for the Riemann zeta function, the quadratic Dirichlet $L$-functions, and the $L$-functions for the cusp forms. For the Riemann zeta function, he showed for large $T$
$$\max_{\substack{T<|t|\le2T}}\big|\zeta(\tfrac12+it)|\ge\exp\bigg(\big(1+o(1)\big)\sqrt{\frac{\log T}{\log_2T}}\bigg).$$
In 2017, Bondarenko and Seip \cite{bondarenko2017Duke} improved this result to
$$\max_{\substack{0<|t|\le T}}\big|\zeta(\tfrac12+it)|\ge\exp\bigg(\big(\tfrac1{\sqrt2}+o(1)\big)\sqrt{\frac{\log T\log_3T}{\log_2T}}\bigg).$$
Their breakthrough was based on the connection between the Riemann zeta function and GCD sums, which was first observed by Aistleitner \cite{aistleitner2016MathAnn}. The constant $\frac1{\sqrt2}$ was subsequently improved to $1$ by  Bondarenko and Seip \cite{bondarenko2018argument}, and then to $\sqrt2$ by La Bret\`eche and Tenenbaum \cite{tenen2019galsum}.

Now let \(\F\) denote the set of all fundamental discriminants
and \(\chi_d := \big(\frac{d}{\cdot} \big)\) be the real primitive character modulo \(|d|\). A fundamental discriminant \(d\) is either square-free with \(d \equiv 1 \pmod 4\), or of the form \(d =4N\), where \(N\) is square-free, and \(N \equiv 2~\text{or}~3 \pmod 4\). Hence, we have
\[
\{8q: q~ \text{is an odd prime}\} \subset \F.
\]
For quadratic Dirichlet $L$-functions at the central point, Soundararajan \cite[Theorem 2]{soundararajan2008mathann} showed that, for sufficiently large $X$,
$$ 
\max_{\substack{X<|d|\le2X \\ d \in \F}}\big|L(\tfrac12,\chi_d)|\ge\exp\bigg(\big(\tfrac1{\sqrt5}+o(1)\big)\sqrt{\frac{\log X}{\log_2X}}\bigg).
$$
Fan, Hua, and Xie \cite{FHX2026qjmath} extended this result to the family of quadratic Dirichlet \(L\)-functions with prime moduli. They proved that
$$
\max_{\substack{X<|q|\le2X \\q \equiv 1 \pmod 8 }}\big|L(\tfrac12,\chi_q)|\ge\exp\bigg(\big(\sqrt{\tfrac8{45}}+o(1)\big)\sqrt{\frac{\log X}{\log_2X}}\bigg).
$$
More recently, under the Generalized Riemann Hypothesis (GRH), Darbar and Maiti \cite{Darbar2025mathann} showed
   $$
   \max_{\substack{X<|d|\le2X \\ d \in \F}}\big|L(\tfrac12,\chi_d)|\ge\exp\bigg((\tfrac{1}{2}+o(1))\sqrt{\frac{\log X\log_3X}{\log_2X}}\bigg).
   $$
In earlier work, the authors \cite{dong2026arxiv} improved the constant $\frac12$ to $1$.
For the family \(\{\chi_{8q}:q~ \text{is an odd prime}\}\), Gao \cite{Gao2026arxiv} proved that
 \[
    \max_{\substack{ X < |q| \le 2X }}\big| L( \tfrac{1}{2},\chi_{8q}) \big| \ge \exp \bgg( \bg(\tfrac12+o(1)\bg)\sqrt{\frac{\log X \lt X}{\lo X}} \bgg).
    \]
The aim of the present paper is to improve Gao's constant $\frac12$ to $1$.
\begin{thm}
    \label{thm1} 
    Assuming GRH. Then, for sufficiently large \(X\),
    \[
    \max_{\substack{ X < |q| \le 2X }}\big| L( \tfrac{1}{2},\chi_{8q}) \big| \ge \exp \bgg( \bg(1+o(1)\bg)\sqrt{\frac{\log X \lt X}{\lo X}} \bgg).
    \]
\end{thm}

\section{Preliminaries}
\label{sec-pre}
In this section, we present several lemmas. We begin with an approximate functional equation for quadratic Dirichlet \(L\)-functions.
\begin{lem}
    \label{appox}
        For any odd prime \(q\), we have
        \[
        L( \tfrac{1}{2},\chi_{8q}) = 2\sum_{n \ge 1} \frac{\chi_{8q}(n)}{\sqrt{n}}\omega\Big( \frac{n}{\sqrt{q}}\Big),
        \]
        where for any real number \(\xi>0\),
        \[
        \omega(\xi) := \frac{1}{2\pi\mmi}\int_{(2)} \Bg(\frac{8}{\pi}\Bg)^{s/2}\frac{\Gamma(s/2+1/4)}{\Gamma(1/4)}\xi^{-s}\frac{\mmd s}{s}.
        \]
        The function $\omega(\xi)$ is real-valued and smooth on $(0, \infty)$. We have $\omega(\xi)=1+O(\xi^{1/2-\varepsilon})$ as $\xi \to 0^+$, and $\omega^{(j)}(\xi) \ll \mme^{-\xi}$ for \(j \ge 0\) as $\xi \to \infty$. Moreover, \(\omega(\xi)>0\) and \(\omega^\prime(\xi) <0 \) for \(\xi>0\).
\end{lem}
\begin{proof}
    This comes from \cite[Lemma 2.1]{Gao2026arxiv}, and it follows from \cite[Lemmas 2.1 and 2.2]{soundararajan2000annmath}; see also \cite[Lemma 7]{Darbar2025mathann}.
\end{proof}

As in \cite{soundararajan2008mathann,Gao2026arxiv}, let \(\Phi\) denote a smooth function compactly supported in \([1,2]\) satisfying \(\Phi(t) \in [0,1]\) and \(\Phi(t)=1\) for \(t \in [5/4,7/4]\). For any complex number \(s\), let \(\whp(s)\) denote the Mellin transform of \(\Phi\), defined by
\[
\whp(s) := \int_0^\infty \Phi(x)x^s\frac{\mmd x}{x}.
\]
We write \(n=\square\) to indicate that \(n\) is a perfect square. Throughout, \(\varepsilon>0\) denotes an arbitrarily small constant, not necessarily the same at each occurrence. We shall use the following conditional estimate for sums of quadratic characters.
\begin{lem}
    \label{charsum}
    Assuming GRH. Let \(c\) be a positive odd integer. Then for sufficiently large \(X\), we have
    \[
    {\sum_{q}}^\ast (\log q)\chi_{8q}(c) \whp\Bg(\frac{q}{X}\Bg) = \mathds{1}_{c=\square} \whp(1)X+O\big(X^{1/2+\varepsilon}\log(c+2)\big).
    \]
    Here, $\mathds{1}_{n=\square}$ indicates the indicator function of the set of perfect squares. and \({\sum_{q}}^\ast \) denotes a sum over odd primes \(q\), that is, \((q,2)=1\).
\end{lem}
\begin{proof}
    This is \cite[Lemma 2.4]{Gao2026arxiv}.
\end{proof}

Let \((m,n)\) and \([m,n]\) denote the greatest common divisor and the least common multiple of positive integers \(m\) and \(n\), respectively. The following result for GCD sums plays a key role in the proof of Theorem \ref{thm1}.
\begin{lem}
\label{GCD}
    Let $\mcm$ be any set of positive square-free integers with $|\mcm|=N$. Then as $N\to\infty$, we have
   $$\max_{|\mcm|=N}\sum_{m,n\in\mcm}\sqrt{\frac{(m,n)}{[m,n]}}=N\exp\bigg((2+o(1))\sqrt{\frac{\log N\log_3N}{\log_2N}}\bigg).$$
\end{lem}
\begin{proof}
    This is \cite[Eq. (1.5)]{tenen2019galsum}.
\end{proof}
Moreover, if $P_+(n)$ denotes the largest prime divisor of $n$, the set \(\mcm\) is chosen so that
$$y_\mcm :=\max_{m\in\mcm}P_+(m)\le (\log N)^{1+o(1)}.$$

\section{Proof of Theorem \ref{thm1}}
\label{sec-proof-thm1}

Let \(X\) be large, and let \(N=\lfloor X^{1/4-\delta} \rfloor\) , where \(\delta>0\) is fixed and sufficiently small. Choose a set \(\mcm\) as in Lemma \ref{GCD}, and define the resonator 
\[
R_q = \sum_{m\in \mcm} \chi_{8q}(m)
\]
for each odd prime \(q\). Furthermore, to apply the resonance method, define
\begin{align*}
        S_1 &: = S_1(R_q,X) ={\sum_{q}}^\ast (\log q) R_q^2 \whp\Bg(\frac{q}{X}\Bg), \\
        S_2 &: = S_2(R_q,X) ={\sum_{q}}^\ast (\log q) L\Bg(\frac{1}{2},\chi_{8q}\Bg)R_q^2 \whp\Bg(\frac{q}{X}\Bg).
\end{align*}
Since \(L(1/2,\chi_{8q}) \in \mbr\), we have
\begin{align}
    \label{max1}
    \max_{X <q \le 2X}|L(\tfrac{1}{2},\chi_{8q})| \ge \frac{S_2}{S_1}.
\end{align}

We now derive a lower bound for the ratio \(S_2/S_1\) by estimating \(S_1\) and \(S_2\), respectively. Expanding \(R_q^2\) in the definition of \(S_1\), we obtain
\[
S_1={\sum_{q}}^\ast (\log q) R_q^2 \whp\Bg(\frac{q}{X}\Bg) = \sum_{m,n \in \mcm}{\sum_{q}}^\ast (\log q) \chi_{8q}(mn) \whp\Bg(\frac{q}{X}\Bg).
\]
Applying Lemma \ref{charsum} gives
\[
S_1=\sum_{m,n \in \mcm} \whp(1)X \mathds{1}_{mn=\square} + O \Bg(\sum_{m,n \in \mcm}X^{1/2+\varepsilon} \log(mn+2)\Bg).
\]
For any \(m \in \mcm\), the prime number theorem implies that
\begin{align}
    \label{logmupper}
    \log m \le \sum_{p \le y_\mcm}\log p \sim y_\mcm \le (\log N)^{1+o(1)}.
\end{align}
Thus, combining \(|\mcm| = N=\lfloor X^{1/4-\delta} \rfloor \) with \eqref{logmupper}, we deduce that \(\log(mn+2) \ll X^\varepsilon\), and 
\[
S_1= \sum_{\substack{m,n\in\mcm \\ mn=\square}}\whp(1)X +O\big(X^{1/2+2\varepsilon}N^2 \bg).
\]
As \(\mcm\) consists of square-free numbers, \(mn=\square\) implies \(m=n\). Hence,
\begin{align}
    \label{S1upper}
    S_1 = \whp(1)XN + O\big(X^{1/2+2\varepsilon}N^2 \bg) \le \bg(1+o(1)\bg)\whp(1)X.
\end{align}

We next turn our attention to \(S_2\). Similarly, expanding \(R_q^2\) and applying Lemma \ref{appox} to \(L(\tfrac12,\chi_{8q})\), we obtain
\begin{align*}
    S_2 &={\sum_{q}}^\ast (\log q) L\Bg(\frac{1}{2},\chi_{8q}\Bg)R_q^2 \whp\Bg(\frac{q}{X}\Bg) \\ 
    &= 2\sum_{m,n\in \mcm} \sum_{k \ge 1}\frac{1}{\sqrt{k}} {\sum_{q}}^\ast (\log q)\chi_{8q}(kmn)\whp\Bg(\frac{q}{X}\Bg)\omega\Bg(\frac{k}{\sqrt{q}}\Bg).
\end{align*}
For the innermost sum, using Lemma \ref{charsum} and the definition of \(\Phi\), we get
\begin{align*}
    &\quad{\sum_{q}}^\ast (\log q)\chi_{8q}(kmn)\whp\Bg(\frac{q}{X}\Bg)\omega\Bg(\frac{k}{\sqrt{q}}\Bg) \\
    &= \int_X^{2X} \omega\Bg(\frac{k}{\sqrt{q}}\Bg) \mmd \bg(\whp(1)t\mathds{1}_{kmn=\square} + O\bg(t^{1/2+\varepsilon} \log(kmn+2)\bg) \\
    &=\whp(1)X\mathds{1}_{kmn=\square} \int_1^2\omega\Bg(\frac{k}{\sqrt{Xu}}\Bg)\mmd u+\mce,
\end{align*}
where the error term \(\mce\) satisfies
\[
\mce \ll X^{1/2+\varepsilon}\log(kmn+2)\Bg(\Bg|\omega\Bg(\frac{k}{\sqrt{X}}\Bg) \Bg| +\Bg|\omega\Bg(\frac{k}{\sqrt{2X}}\Bg) \Bg|+\int_X^{2X}\Bg| \omega^\prime\Bg(\frac{k}{\sqrt{t}}\Bg)\Bg|\frac{1}{2t^{3/2}}\mdt\Bg).
\]
By Lemma \eqref{charsum}, \(\omega(\xi)\) and \(\omega^\prime(\xi)\) decay exponentially. We may therefore restrict the sum over \(k\) to \(k \le X^{1/2+\varepsilon}\), which yields
\[
\mce \ll X^{1/2+\varepsilon}\log(kmn+2) \ll X^{1/2+\varepsilon}.
\]
Here, we use \eqref{logmupper} again. We then have
\begin{align*}
    S_2 = \whp(1)X\sum_{m,n\in \mcm} \sum_{\substack{k \ge 1 \\ kmn=\square}} \frac{1}{\sqrt{k}}\int_1^2\omega\Bg(\frac{k}{\sqrt{Xu}}\Bg)\mmd u +O\Bg(X^{1/2+\varepsilon} \sum_{k \le X^{1/2+\varepsilon}}\frac{1}{\sqrt{k}}\sum_{m,n\in\mcm}1\Bg).
\end{align*}
Trivially, partial summation shows that
\[
\sum_{k \le X^{1/2+\varepsilon}}\frac{1}{\sqrt{k}} \ll X^{1/4+\varepsilon}.
\]
Therefore, using \(|\mcm| = N\), we obtain
\begin{align}
    \label{s2eq}
    S_2 = \whp(1)X\sum_{m,n\in \mcm} \sum_{\substack{k \ge 1 \\ kmn=\square}} \frac{1}{\sqrt{k}}\int_1^2\omega\Bg(\frac{k}{\sqrt{Xu}}\Bg)\mmd u +O\bg(X^{3/4+2\varepsilon}N^2\bg).
\end{align}

Let \(\mcd\) denote the main term on the right-hand side of \eqref{s2eq}. Since \(\omega(\xi) \ge 0\), we may retain only those terms satisfying
\[
k = \frac{[m,n]}{(m,n)}
\]
and obtain the following crude lower bound for \(\mcd\):
\begin{align}
    \label{Dlower1}
    \mcd \ge \whp(1)X \sum_{m,n\in \mcm}\sqrt{\frac{(m,n)}{[m,n]}} \int_1^2\omega\Big(\frac{[m,n]}{(m,n)\sqrt{Xu}} \Big)\mmd u .
\end{align}
To employ the asymptotic formula of \(\omega(\xi)\) when \(\xi\) is small given in Lemma \ref{appox}, we restrict the sum to \([m,n]/(m,n) \le X^\varepsilon\) and obtain
\[
\mcd \ge \whp(1)X\sum_{\substack{m,n\in\mcm \\ [m,n]/(m,n) \le X^\varepsilon}} \sqrt{\frac{(m,n)}{[m,n]}} \big(1+O\big((X^{-1/2+\varepsilon})^{1/2-\varepsilon}\big)\big).
\]
Following the argument in \cite[p. 25]{tenen2019galsum}, we have, for each fixed \(m \in \mcm\),
\[
\sum_{n \in \mcm}\Big( \frac{(m,n)}{[m,n]} \Big)^{1/3} \le \prod_{p \le y_\mcm}\Big(1+\frac{2}{p^{1/3}-1}\Big) \ll \exp \big(y_\mcm^{2/3}\big).
\]
Rankin's trick then gives the following estimate for the restricted GCD sum:
\begin{align*}
    \sum_{\substack{m,n\in\mcm \\ [m,n]/(m,n) \le X^\varepsilon}} \sqrt{\frac{(m,n)}{[m,n]}} &= \Big(\sum_{m,n\in\mcm} - \sum_{\substack{m,n\in\mcm \\ [m,n]/(m,n) > X^\varepsilon}}\Big) \sqrt{\frac{(m,n)}{[m,n]}} \nonumber \\
    & \ge \sum_{m,n\in\mcm}  \sqrt{\frac{(m,n)}{[m,n]}} -X^{-\varepsilon/6}\sum_{m,n\in\mcm} \Big(\frac{(m,n)}{[m,n]}\Big)^{1/3}\nonumber \\
    & \gg\sum_{m,n\in\mcm}  \sqrt{\frac{(m,n)}{[m,n]}} -X^{-\varepsilon/6}|\mcm| \exp \big(y_\mcm^{2/3}\big).
\end{align*}
Then, combining the fact that \(y_\mcm \le  (\log N)^{1+o(1)}\) and Lemma \ref{GCD}, we deduce that
\begin{align}
    \label{gcdlower}
 \sum_{\substack{m,n\in\mcm \\ [m,n]/(m,n) \le X^\varepsilon}} \sqrt{\frac{(m,n)}{[m,n]}} \ge N  \exp\bigg(\bg(2+o(1)\bg)\sqrt{\frac{\log N\log_3N}{\log_2N}}\bigg).
\end{align}
Substituting \eqref{gcdlower} into \eqref{Dlower1}, we get 
\begin{align}
    \label{Dlower2}
    \mcd \ge \bg(1+o(1)\bg)\whp(1)XN  \exp\bigg(\bg(2+o(1)\bg)\sqrt{\frac{\log N\log_3N}{\log_2N}}\bigg).
\end{align}

Combining \eqref{S1upper}, \eqref{s2eq} and \eqref{Dlower2} shows that
\[
\frac{S_2}{S_1} \ge \exp\bigg((2+o(1))\sqrt{\frac{\log N\log_3N}{\log_2N}}\bigg).
\]
Then, inserting this lower bound for \(S_2/S_1\) into \eqref{max1}, we have
\begin{align*}
    \max_{ X < |q| \le 2X}| L( \tfrac{1}{2},\chi_{8q}) | & \ge\exp\bigg((2+o(1))\sqrt{\frac{\log N\log_3N}{\log_2N}}\bigg) \\
    & \ge  \exp\bigg(\bigg(2\sqrt{\frac{1}{4}-\delta}+o(1)\bigg)\sqrt{\frac{\log X\log_3X}{\log_2X}}\bigg).
\end{align*}
Choosing \(\delta \to 0+\), we complete the proof of Theorem \ref{thm1}.

	\section*{Acknowledgements}
	Z. Dong is supported by the National
	Natural Science Foundation of China (Grant No. 	1240011770). W. Wang is supported by the National
	Natural Science Foundation of China (Grant No. 1250012812). H. Zhang is supported by the Fundamental Research Funds for the Central Universities (Grant No. 531118010622), the National
	Natural Science Foundation of China (Grant No. 1240011979) and the Hunan Provincial Natural Science Foundation of China (Grant No. 2024JJ6120).
	
	\bibliographystyle{siam}
    \bibliography{reference}
\end{document}